\documentclass[a4paper,11pt,reqno,noindent]{amsart}
\usepackage[centertags]{amsmath}
\usepackage{mathtools}
\usepackage{amsfonts,amssymb,amsthm,dsfont,cases,amscd,esint,enumerate}
\usepackage[T1]{fontenc}
\usepackage[english]{babel}
\usepackage[applemac]{inputenc}
\usepackage{newlfont,accents}
\usepackage{color}
\usepackage[body={15cm,21.5cm},centering]{geometry} 
\usepackage{fancyhdr}
\usepackage{enumerate}

\theoremstyle{plain}
\newtheorem{theor10}{Theorem}

\newtheorem{prop10}{Proposition}

\newtheorem{cor10}{Corollary}

\newtheorem{theor0}{Theorem}[section]
\newenvironment{theor}
  {\pushQED{\qed}\begin{theor0}}
  {\popQED\end{theor0}}
\newtheorem{lem0}[theor0]{Lemma}
\newenvironment{lem}
  {\pushQED{\qed}\begin{lem0}}
  {\popQED\end{lem0}}
\newtheorem{prop0}[theor0]{Proposition}
\newenvironment{prop}
  {\pushQED{\qed}\begin{prop0}}
  {\popQED\end{prop0}}
\newtheorem{cor0}[theor0]{Corollary}
\newenvironment{cor}
  {\pushQED{\qed}\begin{cor0}}
  {\popQED\end{cor0}}
\newtheorem{propr0}[theor0]{Property}

\newtheorem{hyp0}[theor0]{Hypothesis}
\newenvironment{hyp}
  {\pushQED{\qed}\begin{hyp0}}
  {\popQED\end{hyp0}}
\newtheorem{result0}[theor0]{Result}

\newtheorem{conj0}[theor0]{Conjecture}

\newtheorem{heur0}[theor0]{Heuristics}

\theoremstyle{definition}
\newtheorem{defin0}[theor0]{Definition}
\newenvironment{defin}
  {\pushQED{\qed}\begin{defin0}}
  {\popQED\end{defin0}}
\newtheorem{rems0}[theor0]{Remarks}

\newtheorem{ex0}[theor0]{Example}

\newtheorem{exs0}[theor0]{Examples}

\newtheorem{rem0}[theor0]{Remark}

\newtheorem{qu0}[theor0]{Question}

\newtheorem{qus0}[theor0]{Questions}

  \newtheorem{as0}[theor0]{Assumption}

\mathchardef\emptyset="001F
\numberwithin{equation}{section}

\newcommand{\ru}{r_\star}
\newcommand{\ra}{r_\diamond}
\newcommand{\rc}{r_\spadesuit}
\newcommand{\rs}{r_\clubsuit}
\newcommand{\Br}{B_\star}
\newcommand{\Ba}{B_\diamond}

\newcommand{\N}{\mathbb N}

\newcommand{\e}{\varepsilon}
\newcommand{\Log}{|\!\log\e|}

\newcommand{\calQ}{\mathcal{Q}}

\newcommand{\Cc}{\mathcal{C}}
\newcommand{\calC}{\mathcal{C}}

\newcommand{\R}{\mathbb R}

\newcommand{\Nc}{\mathcal N}

\newcommand{\loc}{{\operatorname{loc}}}
\newcommand{\Id}{\operatorname{Id}}

\newcommand{\E}{\mathbb{E}}

\newcommand{\diam}{\operatorname{diam}}

\newcommand{\step}[1]{\noindent \textit{Step} #1.}

\newcommand{\Pm}{\mathbb{P}}

\newcommand{\expec}[1]{\mathbb{E}\left[ #1 \right]}

\newcommand{\var}[1]{\mathrm{Var}\left[#1\right]}

\newcommand{\cov}[2]{\operatorname{Cov}\left[{#1};{#2}\right]}

\newcommand{\dTV}[2]{\operatorname{d}_{\operatorname{TV}}\left({#1},{#2}\right)}

\newcommand{\dWW}[2]{\operatorname{W}_2\left({#1},{#2}\right)}

\newcommand{\ura}{\underline{r}_{\diamond}}
\newcommand{\tra}{\tilde{r}_{\diamond}}
\newcommand{\pa}{p_\diamond}

\usepackage{color}
\usepackage[colorlinks,citecolor=black,urlcolor=black]{hyperref}

\usepackage[normalem]{ulem}

\title{Quantitative homogenization for log-normal coefficients}
\author[N. Clozeau]{Nicolas Clozeau}
\author[A. Gloria]{Antoine Gloria}
\author[S. Qi]{Siguang Qi}

\address[Nicolas Clozeau]{ISTA, Klosterneuburg, Austria}
\email{nicolas.clozeau@ist.ac.at}
\address[Antoine Gloria]{Sorbonne Universit\'e, CNRS, Universit\'e de Paris, Laboratoire Jacques-Louis Lions (LJLL), F-75005 Paris, France \& Institut Universitaire de France (IUF) \& Universit\'e Libre de Bruxelles, D\'epartement de Math\'ematique, Brussels, Belgium}
\email{antoine.gloria@sorbonne-universite.fr}
\address[Siguang Qi]{Sorbonne Universit\'e, CNRS, Universit\'e de Paris, Laboratoire Jacques-Louis Lions (LJLL), F-75005 Paris, France}
\email{siguang.qi@sorbonne-universite.fr}

\begin{document}
\selectlanguage{english}

\begin{abstract}
We establish quantitative 
homogenization results for the popular log-normal coefficients. Since the coefficients are neither bounded nor uniformly elliptic, standard proofs do not apply directly. Instead, we take inspiration from the approach developed for the nonlinear setting by the first two authors and capitalize on large-scale regularity results by Bella, Fehrmann, and Otto for degenerate coefficients in order to leverage an optimal control (in terms of scaling and stochastic integrability) of oscillations and fluctuations.
\\ \\
MSC: 35R60, 35B27, 35B65, 60F05, 60H07
\\
Keywords: quantitative stochastic homogenization, log-normal coefficients, regularity theory, fluctuations
\end{abstract}

\maketitle
\tableofcontents

\section{Main results}

\subsection{Degenerate coefficients in stochastic homogenization}

Quantitative stochastic homogenization for linear elliptic PDEs with bounded and uniformly elliptic coefficients is by now well-established, either based on functional calculus and nonlinear concentration of measure \cite{GO1,GO2,Gloria-Otto-10b,GNO1,GNO2,DG1,DG2,GNO-reg,GNO-quant,DGO1,DGO2,DO-20,duerinckx2019scaling} or based on renormalization and linear concentration of measure \cite{AS,Armstrong-Mourrat-16,AKM1,GO4,AKM2,AKM-book}.
In both cases, the theory assumes non-degeneracy of the coefficients. Whereas there is no doubt the theory should extend to mildly degenerate coefficients, only quantitative homogenization on the percolation cluster \cite{AD-18,Dario} and for rigid inclusions \cite{MR4468195} has been established so far.

\medskip

In geology (and more precisely for applications to oil or water recovery, and $CO_2$ storage), permeability can range over more than 8 orders of magnitude.
Such permeability fields are often modelled using log-normal distributions (that is, they are obtained as the exponential of a Gaussian random field -- failing boundedness and uniform ellipticity). When one is interested in the behavior of the system at the scale of the correlation-length, this is the realm of expansions \`a la Karhunen-Lo\`eve, see e.~g.~\cite{MR3601011,MR4570337}.
When one is interested in the large-scale behavior of the system, this is the realm of (numerical) homogenization. 
It is not a coincidence that the very first article on numerical homogenization \cite{MR1455261} in the applied mathematics community (which triggered a long-lasting activity in multiscale modeling and simulations, e.~g.~\cite{MR2477579}) deals with log-normal coefficients.
Although the numerical analysis of such methods for random coefficients was the main motivation 
to develop a quantitative stochastic homogenization theory in  \cite{GO1,GO2,MR2991032,GNO1,Gloria-12}, the case of degenerate coefficients was not covered.
The aim of this article is to fill this gap, and extend the quantitative homogenization theory 
 based on functional calculus and nonlinear concentration of measure to log-normal coefficients. 
The results of this work were announced in the plenary talk of the second author at the SIAM conference on Mathematical Aspects of Materials Science (MS21), and extend \cite{GQ-24} to higher dimensions.

\medskip

Quantitative homogenization starts with regularity theory -- like the perturbative Meyers' estimate (which provides a higher integrability result). For bounded and uniformly elliptic coefficients, such estimates are standard and deterministic. For degenerate coefficients however, one cannot hope for a deterministic version of these estimates, and we only expect large-scale versions -- that is, such estimates only hold at a random scale onwards.
Our strategy to get around this problem is inspired by the work \cite{MR4619005} of the first two authors on the homogenization of genuinely nonlinear monotone operators. In this setting the linearized equation is an elliptic equation with coefficients that involve the nonlinear corrector gradient, and therefore are unbounded (and thus degenerate in the usual sense of the word in homogenization). In \cite{MR4619005} we established large-scale Meyers' estimates using quantitative homogenization itself (and in particular the decay of spatial averages of correctors). 
In the present work, the problem is somehow simpler because the degeneracy of the coefficients is a given datum rather than an unknown of the problem (as opposed to the integrability of the nonlinear correctors in \cite{MR4619005}). 
The main insight of the present work are the large-scale Meyers estimates we establish
in Theorem~\ref{th:Meyers} and the control of correctors in Theorem~\ref{th:bd-corr}, that we prove following the strategy laid out in \cite{MR4619005} for a nonlinear problem.
From there, we capitalize on  \cite{bella2018liouville} and on by-now standard methods 
to leverage a complete quantitative homogenization theory, that we state for completeness with precise references to the literature. In particular, we give a complete description of fluctuations of observables of the (random) solution, in the spirit of uncertainty quantification. 

\medskip

For simplicity of the exposition, we consider the coefficient field $a(x)\Id$ where $a(x)=\exp(G(x))$ is a scalar field obtained as the exponential of a Gaussian field $G$ (this specific form is not essential, and encapsulates the main difficulty of unboundedness and degeneracy). 
We further assume that this Gaussian field has integrable covariance (this is essential for the upcoming proof).
Since our strategy is based on functional calculus and nonlinear concentration, one can also treat Poissonian models. A typical example would be as follows: consider a Poisson point process, and the associated Voronoi tessellation of $\R^d$. Define the random coefficients as $a(x)\Id$, where $a(x)$ denotes the diameter of the Voronoi cell containing $x$ -- we have $\inf a = \inf a^{-1}=0$. Since $a$ satisfies a multiscale spectral gap inequality (with exponential weight) in the sense of \cite{DG1,DG2},  our strategy still holds with minor modifications (see e.g.~\cite{GNO1,GNO2} where both Gaussian and Poissonian fields are considered).

\medskip

To conclude, let us comment on the interest of multiscale functional inequalities in the sense of \cite{DG1,DG2} in this context. Quantitative stochastic homogenization requires two ingredients on the coefficient field: a rate for the convergence of spatial averages to their expectation, and concentration of measures (that is, the stochastic integrability).
Whereas (nonlinear) concentration depends little on the covariance function when the coefficient field satisfies a functional inequality, (linear) concentration degrades very fast with correlations when it is obtained by (linear) mixing conditions (such as alpha-mixing). In particular, as soon as one is interested in coefficient fields whose correlations (in a suitable mixing sense) do not decay at a super-algebraic rate, functional inequalities seem unavoidable.

\medskip

In this article we shall consider the following class of coefficient fields.
\begin{hyp}\label{hypo}
Let $G:\R^d \to \R$ be a Gaussian field with integrable covariance function $\calC$
which is $2\gamma$-H\"older continuous at 0 for some $0<\gamma<\frac12$.
We set $a(x):=\exp(G(x))$.
\end{hyp}
All the upcoming results assume Hypothesis~\ref{hypo}.

\subsection{Perturbative large-scale regularity}

We start by defining a scale at which the log-normal field behaves like a uniformly and elliptic coefficient field.
\begin{prop}\label{prop:mombd-ra}
Let $a$ be as in Hypothesis~\ref{hypo}.
Set $p_\diamond =d+ 1$. There exists a stationary $\frac18$-Lipschitz field $\ra$ such that for all $x\in\mathbb{R}^d$ and $r \ge \ra(x)$ 
\[
\tfrac1{2}\expec{a^{p_\diamond}+a^{-p_\diamond}} \le\fint_{B_r(x)} a^{p_\diamond}+a^{-p_\diamond} \le 2 \expec{a^{p_\diamond}+a^{-p_\diamond}},
\]
and which satisfies 
\begin{equation}\label{MomentNonSupEllRatio}
\expec{\exp(\frac1C \log^2(1+\ra))} \le 2,
\end{equation}
for some $C>0$.
In what follows, for all $x\in \R^d$ we set $\Ba(x):=B_{\ra(x)}(x)$.

Let $0<\e \le 1$, we define for all $R\ge 1$
\[
\rs(R):= \Big(R^{-\frac \e2} \sup_{B_R} (a+a^{-1})\Big)^2.
\]
There exists $C>0$ (depending on $\e$) such that 
\begin{equation}\label{MomentBoundEllipticityRadius}
\sup_{R\ge 1}\, \expec{\exp(\frac1C \log^2(1+\rs(R)))} \le 2.
\end{equation}
\end{prop}
Based on this (which we prove in the appendix), we shall establish the following large-scale Meyers estimate.
\begin{theor}[Quenched Meyers' estimates]\label{th:Meyers}
There exists $\kappa>0$ (depending on $d$) such that for all $2 \le p \le 2+ \kappa$, and all functions $u,h$ satisfying 
in the weak sense in $\R^d$
$$
 -\nabla \cdot a \nabla u = \nabla \cdot \sqrt ah, 
$$
we have
\begin{equation}
\int_{\R^d} \Big(\fint_{\Ba(x)}a|\nabla u |^2\Big)^\frac p2 dx 
\,\lesssim_p\,  \int_{\R^d} \Big(\fint_{\Ba(x)}  |h|^2\Big)^\frac p2 dx .
\end{equation}
\end{theor}
A consequence of the proof of the Meyers estimates is the following large-scale hole-filling estimate.
\begin{cor}[Hole-filling estimate in the large]\label{Lholefilling}
There exists $0<\beta \le d$ with the following property.
Let $R \ge \ra$, and let $u$ satisfy
\[
-\nabla\cdot a\nabla u= 0 \text{ in $B_R$}.
\]
Then for all $\ra\le r \le R$, we have
\begin{equation}
\fint_{B_r}a |\nabla u|^2  \lesssim (\tfrac{R}{r})^{d-\beta}\fint_{B_R}a|\nabla u|^2  .
\label{Lholefillingesti}
\end{equation}
\end{cor}
Since by Proposition~\ref{prop:mombd-ra} the stationary field $\ra$ has finite super-algebraic moments, one may upgrade the above into  annealed estimates, following an argument of \cite{DO-20}.
\begin{theor}[Annealed Meyers' estimate]\label{th:annealedmeyers}
There exists $\kappa>0$ such that for all  $|p-2|,|m-2| \le \kappa$ and $0< \delta \le \frac12$, and all functions $u,h$ satisfying 
in the weak sense in $\R^d$
$$
-\nabla \cdot a \nabla u = \nabla \cdot \sqrt a h, 
$$
we have
\begin{equation*}
\int_{\R^d} \expec{\Big(\fint_{B(x)}  a|\nabla u|^2\Big)^\frac p2}^\frac mp dx \,
\lesssim_p\, \delta^{-\frac14} |\log \delta |^\frac12  \int_{\R^d} \expec{\Big(\fint_{B(x)}  |h|^2\Big)^\frac {p(1+\delta)}2 }^\frac m{p(1+\delta)} dx.
\end{equation*}
\end{theor}
These results are proved in Section~\ref{sec:Meyers}.

\subsection{Minimal radius and bounds on correctors}

Based on these perturbative large-scale regularity estimates and a buckling argument making use of the CLT scaling, we establish bounds on correctors, which are defined in this degenerate context in \cite[Lemma~1]{bella2018liouville}.
\begin{lem}
There exist two tensor fields $\{\phi_i\}_{1\le i \le d}$ and $\{\sigma_{ijk}\}_{1\le i,j,k\le d}$
with the following properties. The gradient fields are stationary, have bounded moments, and are of vanishing expectation: $\expec{\nabla \phi_i}=\expec{\nabla \sigma_{ijk}}=0$ and for all $1\le p<\infty$, 
\[
\sum_{i=1}^d \expec{a |\nabla \phi_i|^2}+\sum_{i=1}^d \expec{|\nabla \phi_i|^\frac{2p}{p+1}}^\frac{p+1}{2p} + \sum_{i,j,k=1}^d \expec{|\nabla \sigma_{ijk}|^\frac{2p}{p+1}}^\frac{p+1}{2p} \,\lesssim \, \expec{a^p+a^{-p}}.
\]
The field $\sigma$ is skew-symmetric in its last two indices, that is, 
\[
\sigma_{ijk}=-\sigma_{ikj}.
\]
Furthermore, for $\expec{\cdot}$-a.e. $a$ we have
 \begin{equation}\label{CorectorEquation}
-\nabla\cdot a (\nabla \phi_i+e_i)=0,
\end{equation}
and\footnote{the divergence is taken wrt the last index}
\[
q_i:=a(\nabla\phi_i+e_i)=\expec{q_i} e_i +\nabla \cdot \sigma_i,
\]
together with the gauge equation
 \begin{equation}\label{GaugeEquation}
-\triangle \sigma_{ijk}=[\nabla\times q_i]_{jk}:=\partial_{j} q_{ik}-\partial_k q_{ij}.
\end{equation}
Finally,  the homogenized coefficient $\bar a e_i:=\expec{a(\nabla\phi_i+e_i)}$ is uniformly elliptic in the sense that for 
$K:=\expec{a^{d+1}+a^{-(d+1)}}$ and for all $\xi \in \R^d$,
\[
\frac1K |\xi|^2 \le \xi \cdot \bar a \xi \quad \text{ and }\quad |\bar a \xi|\le K|\xi|.
\]
\end{lem}
An output of \cite[Lemma~1]{bella2018liouville} is that $\nabla \phi_i$ and $\nabla \sigma_{ijk}$ are uniquely defined.
We are in position to introduce the minimal radius.
\begin{defin}\label{DefRstarStatement}
Recall $p_\diamond = d+1$.
We define the stationary field $\tilde \ru:\R \to \R_+$ via
\begin{equation}\label{Def:MinimalRadius}
\tilde \ru:x\mapsto \inf_{r \ge \ra(x)} \Big\{\forall \rho \ge r, \, \frac1\rho\Big( \fint_{B_\rho}  \Big|(\phi,\sigma)-\fint_{B_\rho}(\phi,\sigma)\Big|^\frac{2p_\diamond}{p_\diamond-1}\Big)^\frac{p_\diamond-1}{2p_\diamond}\le \frac1C \Big\},
\end{equation}
and we denote by $\ru$ the smallest $\frac18$-Lipschitz field larger or equal to $\tilde \ru$
(which is also stationary).
In what follows, for all $x\in \R^d$ we set $\Br(x):=B_{\ru(x)}(x)$.
\end{defin}
Our first quantitative result is the following bound on the growth of the correctors.
\begin{theor}\label{th:bd-corr}
For all $x\in \R^d$, 
\begin{equation}\label{MomentPhiItself}
\Big(\fint_{B(x)} \Big|(\phi,\sigma)-\fint_{B(0)}(\phi,\sigma)\Big|^\frac{2p_\diamond}{p_\diamond-1}\Big)^\frac{p_\diamond-1}{2p_\diamond}  \,\le\, C_x \mu_d(|x|),
\end{equation}
where
\begin{equation}\label{mud}
\mu_d(t):=
\left\{
\begin{array}{rcl}
\sqrt{t+1} &:&d=1,
\\
\log(t+2)^\frac12&:&d=2,
\\
1 &:&d>2,
\end{array}
\right.
\end{equation}
and where $C_x$ is a random variable that satisfies for some $0<C<\infty$
\begin{equation}\label{StoInteConstant}
\expec{\exp(\frac1{C}\log^2(1+C_x))}\le 2
\end{equation}
\end{theor}
We also have the following control of the minimal radius.
\begin{theor}\label{th:min-rad}
There exists $0<C<\infty$ such that the minimal radius $\ru$ satisfies
\[
\expec{\exp(\frac1C \log^2(1+\ru))} \le 2.
\]
\end{theor}
These results are proved in Section~\ref{sec:corr}.
As shown in dimension $d=1$ in \cite{GQ-24}, the stochastic integrability in the above results is optimal.

\subsection{Non-perturbative large-scale regularity}

The following non-perturbative large-scale regularity results are post-processings of \cite{bella2018liouville}
using Theorem~\ref{th:min-rad}. (See Section~\ref{sec:LS} for the necessary adaptations.)
\begin{prop}[Large-scale Schauder estimates]\label{SchauderTheoryLarge}
If $u,f,g$ satisfy in the weak sense in $B_R$ for some $R \ge \ru$
$$
-\nabla \cdot a (\nabla u+g) = \nabla \cdot h,
$$
then we have for all $0<\alpha <1$
\begin{multline*}
\sup_{r \in [\ru,R]} \fint_{B_r} (\nabla u +g) \cdot a (\nabla u+g) \,\lesssim\,
\fint_{B_R}  (\nabla u +g) \cdot a (\nabla u+g) 
\\
+\sup_{r \in [\ru,R]} \big(\tfrac R r\big)^{2\alpha} \fint_{B_r} \Big( \big(g-\fint_{B_r}g\big) \cdot a \big(g-\fint_{B_r}g\big)+\big(h-\fint_{B_r}h\big) \cdot a^{-1} \big(h-\fint_{B_r}h\big)\Big).
\end{multline*}
In particular, for $g=h=0$, we have the following mean-value property for $a$-harmonic functions
(or large-scale Lipschitz property): For all $\ru\leq r\leq R$,
\[
\fint_{B_r}\nabla u\cdot a\nabla u\lesssim \fint_{B_R}\nabla u\cdot a\nabla u.\]
\end{prop}
Based on the above Lipschitz estimate, \cite[Corollary~4]{GNO-reg} directly yields large-scale Calder\'on-Zygmund estimates.
\begin{theor}[Quenched Calder\'on-Zygmund estimates]\label{th:quenchedCZ}
For all $1<p <\infty$ and for all functions $u,g,h$ satisfying 
in the weak sense in $\R^d$
$$
-\nabla \cdot a (\nabla u+g)=\nabla \cdot h, 
$$
we have
\begin{equation}
\int_{\R^d}  \Big(\fint_{\Br(x)}  \nabla u \cdot a \nabla u \Big)^\frac p2  dx 
\,
\lesssim_{p }\, \int_{\R^d} \Big(\fint_{\Br(x)} g\cdot a g + h \cdot a^{-1} h \Big)^\frac {p}2 dx.
\end{equation}
\end{theor}
As in \cite{DO-20}, one can turn the quenched CZ estimates into the following annealed CZ estimates.
\begin{theor}[Annealed Calder\'on-Zygmund estimates]\label{th:annealedCZ}
For all $1<p,m<\infty$ and $0< \delta \le \frac12$, and for all functions $u,g,h$ satisfying 
in the weak sense in $\R^d$
$$
-\nabla \cdot a (\nabla u+g)=\nabla \cdot h, 
$$
we have
\begin{multline}
\int_{\R^d} \expec{\Big(\fint_{B(x)} a |\nabla u|^2\Big)^\frac p2}^\frac mp dx 
\\
\lesssim_{p,m}\, \delta^{-\frac14} |\log \delta |^\frac12  \int_{\R^d} \expec{\Big(\fint_{B(x)} a|g|^2 + a^{-1}|h|^2\Big)^\frac {p(1+\delta)}2 }^\frac m{p(1+\delta)} dx.
\end{multline}
\end{theor}

\subsection{Quantitative homogenization}

We now have all the necessary tools to make homogenization quantitative.
Let $f\in L^2(\R^d)^d$, and for all $\e>0$, consider the unique Lax-Milgram solution\footnote{Since $a$ is degenerate, there is a small approximation argument needed to establish existence -- see e.g.~Section~\ref{sec:corr}.} $u_\e \in \dot H^1(\R^d)$ of
\begin{equation}\label{e.eq-ueps}
-\nabla \cdot a(\tfrac \cdot \e) \nabla u_\e = \nabla \cdot f,
\end{equation}
as well as the unique Lax-Milgram solution $\bar u \in \dot H^1(\R^d)$ of
the homogenized problem
\begin{equation}\label{e.eq-uhom}
-\nabla \cdot \bar a \nabla \bar u = \nabla \cdot f.
\end{equation}
Our first result characterizes oscillations of $u_\e$ in a strong norm in form of a quantitative two-scale expansion result.
\begin{theor}
For all $\e>0$, $v \in L^2_\loc(\R^d)$, and $x\in \R^d$, set $S_\e(v)(x)= \fint_{B_\e(x)} v$,
and define the two-scale expansion $\bar u_\e^{2s}$ of $u_\e$ as
\[
\bar u_\e^{2s}:= S_\e(\bar u)+\e \phi_i(\tfrac \cdot \e)  S_\e(\partial_i \bar u).
\]
Then for all $p\ge  1$ and $q\ge 1$ we have
\begin{multline*}
\expec{\Big(\int_{\R^d} \Big(\fint_{B_\e(x)} a|\nabla (u_\e-\bar u_\e^{2s})|^2\Big)^\frac p2 dx\Big)^q }^\frac1{pq} \\
\lesssim\, \e \mu_d(\tfrac 1\e) \Big(\int_{\R^d} \mu_d(|x|)^p \Big(\fint_{B_\e(x)} |f(x)|^2\Big)^\frac p2 dx\Big)^\frac1p,
\end{multline*}
where $\mu_d$ is defined in \eqref{mud}.
\end{theor}
The proof is the same as for \cite[Proposition~1]{GNO-quant} based on Theorem~\ref{th:annealedCZ} rather than \cite[Corollary~4]{GNO-reg}.

\medskip

The second set of results quantifies fluctuations of $u_\e$.
More precisely, we study the scaling limits of observables of the field $\nabla u_\e$ and flux $a(\frac \cdot \e)\nabla u_\e$ of the solution. To this aim, we recall the definition of the standard homogenization commutator $\Xi$ and of the commutator of the solution $\Xi_\e(f)$, as introduced in \cite{DGO1,DGO2}.
\begin{defin}
The standard homogenization commutator is the second-order tensor defined for all $1\le i \le d$ by
\[
[\Xi]_i:=(a-\bar a)(\nabla \phi_i+e_i).
\]
The homogenization commutator of the solution $u_\e$ of \eqref{e.eq-ueps}
is defined by
\[
\Xi_\e(f):= (a(\tfrac \cdot \e)-\bar a)\nabla u_\e.
\]
\end{defin}
Our first result is a quantitative two-scale expansion at the level of observables of commutators
in the fluctuation scaling (the so-called pathwise theory of fluctuations),
where we understand an observable as a local average with a test function.
\begin{theor}\label{th:pathwise}
Let $g \in L^2(\R^d)^d$, and denote by $\bar v$ the solution of \eqref{e.eq-uhom} with $f$ replaced by $g$ (and $\bar a$ by its transpose matrix, here $\bar a'=\bar a$).
On the one hand, $\var{\e^{-\frac d2} \int_{\R^d} g\cdot \Xi_\e(f)} \lesssim 1$.
On the other hand, for all $p \ge1$,
\begin{multline*}
\expec{\Big( \e^{-\frac d2} \Big|\int_{\R^d} g\cdot (\Xi_\e(f)-\expec{\Xi_\e(f)}) - \nabla \bar v \cdot \Xi(\tfrac \cdot\e)\nabla \bar u\Big|\Big)^p}^\frac 1p
\\
\lesssim\, \e \mu_d(\tfrac 1\e) \|f \|_{L^4(\R^d,\mu_d^2)}\|g \|_{L^4(\R^d,\mu_d^2)},
\end{multline*}
where $\mu_d$ is defined in \eqref{mud}, and $L^4(\R^d,\mu_d^2)$ is the weighted space with measure $\mu_d^2(x)dx$.
\end{theor}
This result follows from \cite{DO-20,DGO2} using Theorems~\ref{th:bd-corr} and~\ref{th:annealedCZ}. 
A direct post-processing of Theorem~\ref{th:pathwise} allows to recover corresponding results for the field and the flux of the solution (the main two quantities of interest from a physical viewpoint), see \cite[(1.10) \& (1.11)]{DGO2}.

\medskip

To complete the analysis of fluctuations of observables, it remains to investigate the scaling limit of the homogenization commutator. In the case of integrable covariance, the limit is a colored noise, which  follows from the proofs of \cite{DO-20,duerinckx2019scaling} using Theorems~\ref{th:bd-corr} and~\ref{th:annealedCZ}, on top of Malliavin calculus.
We start by a convenient strengthening of the integrability of the covariance function in Hypothesis~\ref{hypo}:
\begin{hyp}\label{hypo:strong}
The covariance function $\calC$ can be decomposed as $\calC=\calC_0\ast \calC_0$ where $\calC_0$ satisfies for some $\beta>d$ \footnote{Note that this decay assumption for $\calC_0$ implies the integrability of $\calC$.}
\begin{equation*}
|\calC_0(x)|\,\le\,C_0(1+|x|)^{-\frac12(d+\beta)},
\end{equation*}
and its Fourier transform $\hat \calC$ is positive almost everywhere.
\end{hyp}
The following quantitative central limit theorem combines the results of \cite[Proposition~3.1, Corollary~4.5, and Theorem~1(ii)]{duerinckx2019scaling}.
\begin{theor}
Assume Hypotheses~\ref{hypo} and~\ref{hypo:strong}.
For all functions $F\in C^\infty_c(\R^d)^{d\times d}$, set $I_\e(F):=\e^{-\frac d2} \int_{\R^d} \Xi(\tfrac \cdot\e):F$.
On the one hand, there exists a non-degenerate constant tensor $\calQ$ of order $4$ such that for all $F,F'\in C^\infty_c(\R^d)^{d\times d}$,
\begin{multline*}
\qquad\Big|\cov{I_\e(F)}{I_\e(F')}-\int_{\R^d} F(x):\calQ:F'(x)\,dx\Big|\\
\,\lesssim_{F,F'}\,\left\{\begin{array}{lll}
\e&:&d>2,\,\beta\ge d+1,\\
\e\Log^\frac12&:&d=2,\,\beta\ge d+1,\\
\e^{\beta-d}&:&d<\beta<d+1.
\end{array}\right.
\end{multline*}
On the other hand, for all $F\in C^\infty_c(\R^d)^{d\times d}$ and $\e>0$,  
\begin{multline*}
\qquad\dWW{\frac{I_\e(F)}{\var{I_\e(F)}^\frac12}}\Nc+\dTV{\frac{I_\e(F)}{\var{I_\e(F)}^\frac12}}\Nc\\
\,\lesssim_F\,\frac1{\var{I_\e(F)}} \e^\frac d2 \exp(C|\log \e|^\frac12),
\end{multline*}
where $\dWW\cdot\Nc$ and $\dTV\cdot\Nc$ denote the $2$-Wasserstein (see e.g.~\cite{NP-book}) and the total variation distance to a standard Gaussian law, respectively.
In particular, with $\sigma^2(F):=\int_{\R^d} F(x):\calQ:F(x)\,dx$, these two estimates combine to 

\begin{multline*}
\qquad\dWW{\frac{I_\e(F)}{\sigma(F)}}\Nc+\dTV{\frac{I_\e(F)}{\sigma(F)}}\Nc
\,\lesssim_F\, \e^\frac d2 \exp(C|\log \e|^\frac12)
\\
+\left\{\begin{array}{lll}
\e&:&d>2,\,\beta\ge d+1,\\
\e\Log^\frac12&:&d=2,\,\beta\ge d+1,\\
\e^{\beta-d}&:&d<\beta<d+1.
\end{array}\right.
\end{multline*}
\end{theor}
The only difference with the case of uniformly elliptic and bounded coefficients is the error estimate for the asymptotic normality: In \cite[Theorem~1(ii)]{duerinckx2019scaling}, $\exp(C|\log \e|^\frac12)$ is replaced by $|\log \e|$. This comes from the optimization argument in Step~6 Proof of Theorem~1(ii) in \cite{duerinckx2019scaling} which involves the stochastic integrability of the corrector gradient -- a similar optimization in the context of log-normal coefficients is worked out in formula (5.20) in Step~2 of the proof of Proposition~5.4 in \cite{GQ-24}, to which we refer the interested reader.

\section{Large-scale Meyers estimates}\label{sec:Meyers}

\subsection{Proof of Theorem~\ref{th:Meyers}}

We follow the standard proof based on a reverse H\"older's inequality (using Caccioppoli's inequality and the Sobolev embedding) and Gehring's inequality.

\medskip

We start with the reverse H\"older inequality.
\begin{lem}[Reverse H\"older]\label{lem:rev-Hol}
Recall that $p_\diamond =d+1$.  For all
$r\ge r_\diamond$ and all $u,h$ related in $B_{2r}$ by
\[
-\nabla \cdot a \nabla u\,=\,\nabla \cdot \sqrt a h,
\]
we have
\[
\fint_{B_r} (\sqrt a|\nabla u|)^2 \,\lesssim \, \fint_{B_{2r}} |h|^2
+\Big(\fint_{B_{2r}}  ({\sqrt a}|\nabla u|)^\alpha \Big)^{\frac2\alpha},
\]
where $1\le \alpha:=\frac{2d(d+1)}{d^2+d+2}<2$.
\end{lem}
\begin{proof}[Proof of Lemma~\ref{lem:rev-Hol}]
We start by establishing a Caccioppoli inequality.
Let $\eta$ be a smooth cut-off for $B_r$ in $B_{2r}$ with $\sup |\nabla \eta|\lesssim r^{-1}$.
Set $\bar u:=\frac{\int \eta^2 a u}{\int \eta^2 a}$, and notice that $\frac1C \int \eta^2 \le \int \eta^2 a \le C \int \eta^2$ since $r\ge r_\diamond$.
Testing the equation for $u$ with $\eta^2(u-\bar u)$, and integrating by parts we obtain
\begin{multline*}
\int \eta^2 a |\nabla (u-\bar u)|^2 +2 \int \eta a (u-\bar u)\nabla \eta \cdot \nabla (u-\bar u)
\\
= \int \eta^2 \sqrt a \nabla (u-\bar u)\cdot h+2 \int \eta \sqrt a(u-\bar u)\nabla \eta \cdot h.
\end{multline*}
Using Cauchy-Schwarz' and Young's inequalities, this entails
\begin{equation}\label{e.pr-mey-1}
 \int \eta^2 a |\nabla (u-\bar u)|^2 
\,\lesssim\, \int \eta^2 |h|^2 + \int |\nabla \eta|^2 a(u-\bar u)^2.
\end{equation}
It remains to deal with the weights on the right-hand side.
We start with the control of $\bar u$.
Set $\bar U:= \fint_{B_{2r}} u$. We have, since $\int_{B_{2r}} (\frac1{|B_{2r}|}- \frac {a\eta^2} {\int a \eta^2})=0$,
\[
\bar U-\bar u := \int_{B_{2r}} (\frac1{|B_{2r}|}- \frac {a\eta^2} {\int a \eta^2}) u=\int_{B_{2r}} (\frac1{|B_{2r}|}- \frac {a\eta^2} {\int a \eta^2}) (u-\bar U).
\]
Hence, by H\"older's inequality with exponents $(d+1,\frac{d+1}d)$ and Poincar\'e-Sobolev' inequality, 
and since $r\ge r_\diamond$
\begin{equation}\label{e.pr-mey-2}
|\bar U-\bar u|\,\lesssim\, \Big(\fint_{B_{2r}} (u-\bar U)^\frac{d+1}d\Big)^\frac d{d+1} \lesssim r\Big(\fint_{B_{2r}} |\nabla u|^\frac{d(d+1)}{d^2+d+1}\Big)^\frac {d^2+d+1}{d(d+1)}.
\end{equation}
Inserting \eqref{e.pr-mey-2}   into \eqref{e.pr-mey-1}, using H\"older's inequality with exponents $(d+1,\frac{d+1}d)$,  and using that $r\ge \ra$, the Poincar\'e-Sobolev inequality, and Jensen's inequality, we obtain 
\begin{eqnarray*}
\fint_{B_r} a |\nabla u|^2 
&\lesssim& \fint_{B_{2r}} \eta^2 |h|^2 + \Big(\fint_{B_{2r}} |\nabla u|^\frac{d(d+1)}{d^2+d+1}\Big)^{2\frac {d^2+d+1}{d(d+1)}}+\int |\nabla \eta|^2 a(u-\bar U)^2
\\
&\lesssim& \fint_{B_{2r}} \eta^2 |h|^2 + \Big(\fint_{B_{2r}} |\nabla u|^\frac{d(d+1)}{d^2+d+1}\Big)^{2\frac {d^2+d+1}{d(d+1)}}+ r^{-2}\Big(\fint_{B_{2r}} |u-\bar U|^\frac{2(d+1)}{d}\Big)^\frac {d+1}d \\
\\
&\lesssim &\fint_{B_{2r}} \eta^2 |h|^2 +  \Big(\fint_{B_{2r}} |\nabla u|^\frac{2d(d+1)}{d^2+2(d+1)}\Big)^\frac{d^2+2(d+1)} {d(d+1)}.
\end{eqnarray*}
It remains to reintroduce $a$.
By H\"older's inequality with exponents $(\frac{d^2+2(d+1)}{d},\frac{d^2+2(d+1)}{d^2+d+2})$,
\begin{eqnarray*}
\Big(\fint_{B_{2r}} |\nabla u|^\frac{2d(d+1)}{d^2+2(d+1)}\Big)^\frac{d^2+2(d+1)} {d(d+1)}
&=&\Big(\fint_{B_{2r}} a^{-\frac{d(d+1)}{d^2+2(d+1)}}|\sqrt{a}\nabla u|^\frac{2d(d+1)}{d^2+2(d+1)}\Big)^\frac{d^2+2(d+1)} {d(d+1)}
\\
&\le&
\Big(\fint_{B_{2r}} a^{-(d+1)} \Big)^\frac{1}{d+1} \Big(\fint_{B_{2r}} |\sqrt{a}\nabla u|^\frac{2d(d+1)}{d^2+d+2}\Big)^\frac{d^2+d+2}{d(d+1)}
\\
&\stackrel{r\ge \ra}\lesssim &\Big(\fint_{B_{2r}} |\sqrt{a}\nabla u|^\frac{2d(d+1)}{d^2+d+2}\Big)^\frac{d^2+d+2}{d(d+1)},
\end{eqnarray*}
from which the claim follows by our choice $\alpha=\frac{2d(d+1)}{d^2+d+2} \in [1,2)$.
\end{proof}
We now recall Gehring's inequality in a form which is convenient for our purposes (see for instance \cite[Theorem 6.38]{giaquinta2013introduction}):
\begin{lem}[Gehring's lemma]\label{gehring}
Let $s>1$, and  let $U$ and $V$ be two non-negative measurable functions in $L^q_\loc(\mathbb{R}^d)$ such that there exists $C>0$ for which for all $r>0$ and $x\in\mathbb{R}^d$
$$
\Big(\fint_{B_r(x)} U^s\Big)^{\frac{1}{s}}\leq C\Big(\fint_{B_{2r}(x)}U+\Big(\fint_{B_{2r}(x)}V^s\Big)^{\frac{1}{s}}\Big).
$$
Then, there exists $\bar{s}> s$ depending on $q$ and $C$ such that for all $r>0$ and $x\in\mathbb{R}^d$, we have 
$$
\Big(\fint_{B_r(x)}U^{\bar{s}}\Big)^{\frac{1}{\bar{s}}}\lesssim \fint_{B_{2r}(x)}U+\Big(\fint_{B_{2r}(x)}V^{\bar{s}}\Big)^{\frac{1}{\bar{s}}}.
$$
\end{lem}
We conclude this paragraph with the proof of Theorem~\ref{th:Meyers}.
\begin{proof}[Proof of Theorem~\ref{th:Meyers}]
We first prove that for all $r>0$ we have
\begin{equation}\label{e.pr-mey-3}
\fint_{B_r} \Big(\fint_{B_\diamond(x)}  a|\nabla u|^2\Big)dx
\,
\lesssim\, \Big(\fint_{B_{2r}} \Big(\fint_{B_\diamond(x)}  a|\nabla u|^2\Big)^\frac \alpha2 dx\Big)^\frac 2\alpha + \fint_{B_{2r}} \fint_{B_\diamond(x)}  |h|^2 dx,
\end{equation}
and notice that the origin plays no role in this estimate.

If $r \le 3 r_\diamond(0)$, the first right-hand side term controls the left-hand side by a covering argument. It remains to address the case $r>3r_\diamond(0)$.
Since $r_\diamond$ is $\frac18$-Lipschitz we have
\[
\fint_{B_r} \Big(\fint_{B_\diamond(x)}  a|\nabla u|^2\Big)dx \,\lesssim\,\fint_{B_{\frac{67}{48}r}}   a|\nabla u|^2,
\]
see \cite[(C7)]{MR4619005}.
We now appeal to the reverse H\"older inequality in  form of
\[
\fint_{B_{\frac{67}{48}r}}   a|\nabla u|^2  \,\lesssim \, \fint_{B_{\frac{17}{12}r}} |h|^2
+\Big(\fint_{B_{\frac{17}{12}r}}  ({\sqrt a}|\nabla u|)^\alpha \Big)^{\frac2\alpha},
\]
which, by a covering argument (see  \cite[(C7)]{MR4619005} again), yields
\[
\fint_{B_{\frac{67}{48}r}}   a|\nabla u|^2  \,\lesssim \, \fint_{B_{2r}} \fint_{\Ba(x)} |h|^2dx
+\Big(\fint_{B_{2r}} \fint_{\Ba(x)}  ({\sqrt a}|\nabla u|)^\alpha dx \Big)^{\frac2\alpha}.
\]
Using Jensen's inequality on the second right-hand side (since $\alpha<2$), 
this proves \eqref{e.pr-mey-3}.
By Lemma~\ref{gehring} applied to
\[
U(x):= \Big(\fint_{B_\diamond(x)}  a|\nabla u|^2\Big)^\frac \alpha 2,
\quad V(x):= \Big(\fint_{B_\diamond(x)}  |h|^2\Big)^\frac \alpha 2,\quad s:=\frac 2 \alpha>1,
\]
one obtains for all $r>0$ and some $\gamma>1$
\begin{equation}\label{e.pr-mey-5}
\Big(\int_{B_r}\Big(\fint_{B_\diamond(x)}  a|\nabla u|^2\Big)^\gamma dx\Big)^\frac1\gamma
\lesssim r^{d(\frac1\gamma-1)}\int_{B_{2r}}\fint_{B_\diamond(x)}  a|\nabla u|^2 dx+\Big(\int_{B_{2r}}\Big(\fint_{B_\diamond(x)}  |h|^2\Big)^\gamma\Big)^{\frac{1}{\gamma}}.
\end{equation}
The desired Meyers' estimate follows by monotone convergence in the limit $r\uparrow +\infty$
combined with the $L^2$-energy estimate
\[
\int_{\R^d} \fint_{\Ba(x)} a|\nabla u|^2dx \lesssim \int_{\R^d} a|\nabla u|^2 \,\lesssim\, \int_{\R^d} |h|^2.
\]
\end{proof}

\subsection{Proof of Corollary~\ref{Lholefilling}}
Let $\gamma$ be as in \eqref{e.pr-mey-5}.
Without loss of generality we can assume $r\ge 3 \ra$ so that by \cite[Lemma~C.2]{MR4619005} we have
for all non-negative functions $h$
\begin{equation}\label{e.partition}
\int_{B_r} h \lesssim \int_{B_{2r}} \fint_{\Ba(x)} h dx, \quad \int_{B_r}\fint_{\Ba(x)} h dx \lesssim \int_{B_{2r}}  h .
\end{equation}
We may also assume $R\ge 4r$.
By \eqref{e.partition} we have
\begin{eqnarray*}
\int_{B_r}a |\nabla u|^2&\lesssim & r^{d} \fint_{B_{2r}} \fint_{\Ba(x)} a |\nabla u|^2 dx
\\
&\le &r^{d} \Big(\fint_{B_{2r}} \Big(\fint_{\Ba(x)} a |\nabla u|^2\Big)^\gamma dx\Big)^\frac1\gamma
\\
&\lesssim & r^{d(1-\frac1 \gamma)}   \Big(\int_{B_{\frac R2}} \Big(\fint_{\Ba(x)} a |\nabla u|^2\Big)^\gamma dx\Big)^\frac1\gamma.
\end{eqnarray*}
By \eqref{e.pr-mey-5} (used with $h\equiv 0$), this entails
\begin{eqnarray*}
\int_{B_r}a |\nabla u|^2&\lesssim & r^{d(1-\frac1 \gamma)}  R^{-d(1-\frac1\gamma)} \int_{B_{\frac R2}} \Big(\fint_{\Ba(x)} a |\nabla u|^2\Big)  dx,
\end{eqnarray*}
and the claim follows from \eqref{e.partition} with $0<\beta=d(1-\frac 1 \gamma)\le d$.

\subsection{Proof of Theorem~\ref{th:annealedmeyers}}

The proof is based on the quenched Meyers estimate in the large of Theorem~\ref{th:Meyers}, on the moment bounds on $\ra$ of Proposition~\ref{prop:mombd-ra} (which allows us to use duality at the price of a loss of stochastic integrability), real interpolation,
and a refined dual version of the Calder\'on-Zygmund lemma due to Shen~\cite[Theorem~3.2]{Shen-07}, based on ideas by Caffarelli and Peral~\cite{CP-98}.
Since it follows the proof of \cite{DO-20} almost line by line and is identical to 
the proof of \cite[Theorem~4.11]{MR4619005}, we leave the details to the reader.

\section{Control of correctors and of the minimal radius}\label{sec:corr}

In this section we prove Theorems~\ref{th:bd-corr} and~\ref{th:min-rad}.
We proceed using an approximation argument, and for all $M\ge 1$, we set $a_M:= (a \wedge M) \vee \frac1M$, which is uniformly elliptic and bounded.
By \cite{GNO1,GNO2}, Theorems~\ref{th:bd-corr} and~\ref{th:min-rad} hold true for all $M<\infty$ (with bounds that depend on $M$ a priori).
By uniqueness of correctors, it is an exercise to show that $\expec{|(\nabla \phi(a_M),\nabla \sigma(a_M))- (\nabla \phi(a),\nabla \sigma(a))|^{2\gamma}} \to 0$ for all $\gamma<1$ as $M\uparrow +\infty$, so that uniform moment bounds on $(\nabla \phi(a_M),\nabla \sigma(a_M))$ 
are retained by $(\nabla \phi(a),\nabla \sigma(a))$ in the limit $M\uparrow+\infty$. The control of $(\phi(a),\sigma(a))$ will  follow similarly from the uniform control of $(\phi(a_M),\sigma(a_M))$, which in turn allows to bound $\ru$.
Since $\ru$ is defined using $L^{\frac{2p_\diamond}{p_\diamond-1}}$ instead of $L^2$, we display the proofs for completeness. 

\medskip

Before we turn to the proofs, notice that Theorems~\ref{th:Meyers} and~\ref{th:annealedmeyers}
hold uniformly with $a$ replaced by $a_M$ for all $M\ge 1$ (with the same constants, exponents, and random radii).
In what follows, it is therefore enough to assume
\begin{itemize}
\item that $a$ is bounded and uniformly elliptic,
\item that Theorems~\ref{th:bd-corr} and~\ref{th:min-rad} hold true for some constants,
\end{itemize}
and to obtain bounds which only depend on Theorems~\ref{th:Meyers} and~\ref{th:annealedmeyers} -- and are therefore uniform wrt $M$.

\subsection{General strategy}

Our main result is as follows.
\begin{theor}[Decay of averages of corrector gradient]\label{th:aver-corr-grad}
There exists $c>0$ depending on $\mathcal{C}(0)$ such that for all $g\in L^2(\R^d)$ and unit vectors $e\in \R^d$
the random field $F:=\int_{\R^d} (\nabla \phi,\nabla \sigma)\cdot g$ satisfies for all $q\ge 1$
\[
\expec{|F|^{2q}}^\frac1q \le  e^{c q} \int_{\R^d} |g|^2.
\]
\end{theor}
Notice that this encodes the CLT scaling. Indeed, for $g_R(x)=|B_R|^{-d} \mathds 1_{B_R}$ this entails $\expec{\Big|\fint_{B_R} (\nabla \phi,\nabla \sigma)\Big|^{2q}}^\frac1q \,\lesssim_q \,R^{-d}$.
Passing from this result on averages of corrector gradients to correctors is routine, see e.g.~\cite{GNO2}.
Passing from correctors to $\ru$ is routine too, see e.g.~\cite{MR4619005}.
We thus only focus on the proof of Theorem~\ref{th:aver-corr-grad}, and follow the strategy of \cite[Section~6]{MR4619005}.

\medskip

We start with a control of averages of correctors that involves the energy density of the corrector.
\begin{prop}\label{prop:averages}
For all $0<\theta<1$, there exists a constant $c>0
$ depending on $\mathcal{C}(0)$ such that for all $g\in L^2(\R^d)$ and unit vectors $e\in \R^d$
the random field $F:=\int_{\R^d} (\nabla \phi,\nabla \sigma)\cdot g$ satisfies for all $q\ge 1$
such that $2q'\le 2+\kappa$ (where $\kappa$ is as in Theorem~\ref{th:annealedmeyers})
\[
\expec{|F|^{2q}}^\frac1q \le e^{cq} \expec{\Big(\int_{B(0)} a(|\nabla \phi|^2+1)\Big)^{q(1+\theta)}}^\frac1{q(1+\theta)} \int_{\R^d} |g|^2.
\]
\end{prop}
Based on Proposition~\ref{prop:averages} itself, we shall prove moment bounds on the energy density of the corrector.
\begin{prop}\label{prop:moment-en-dens} 
There exists a constant $c>0$  depending on $\mathcal{C}(0)$
such that for all $q\ge 1$,
\begin{equation}\label{MomentBoundGradientCor}
\expec{\Big(\int_{B(0)} a(|\nabla \phi|^2+1)\Big)^{q}}^\frac1q \le  e^{cq
},
\end{equation}
and 
\begin{equation}\label{Moment BoundFluxCor}
\expec{\Big(\int_{B(0)}\vert\nabla \sigma\vert^2\Big)^q}^\frac{1}{q} \le  e^{cq}.
\end{equation}
\end{prop}
The combination of Propositions~\ref{prop:averages} and~\ref{prop:moment-en-dens} directly yields Theorem~\ref{th:aver-corr-grad}.
For the proof of Proposition~\ref{prop:moment-en-dens}, we introduce a third minimal radius, which quantifies at which scale the energy density of the corrector 
behaves like a uniformly bounded function of $a$. 
\begin{defin}\label{def:rc}
We define the minimal radius $\rc$ as 
\[
\rc :=\max_{|e|=1}\inf_{r =2^k \ge \ra, k\in \N} \Big\{\forall R \ge r \text{ dyadic}, \quad \fint_{B_{R}} a |\nabla \phi_e|^2 \le C \fint_{B_{2R}}a \Big\},
\]
where $\phi_e$ denotes the corrector in the unit direction $e$, and $C>0$ will be chosen in the proof of Proposition~\ref{prop:moment-en-dens}.
\end{defin}
The core of the proof of Proposition~\ref{prop:moment-en-dens} is a control of the level-sets of $\rc$ using averages of the corrector gradient.

\subsection{Proof of Proposition~\ref{prop:averages}}

Let the direction $e\in \R^d$ with $|e|=1$, and $ij$ be fixed,
and consider $F_1 := \int_{\mathbb{R}^d} \nabla \phi \cdot g$, and $F_2:=\int_{\mathbb{R}^d} \nabla \sigma_{ij} \cdot g$. For background on Malliavin calculus and functional calculus, we refer the reader to \cite[Section~2]{GQ-24} (dimension 1) and \cite[Section~2.1]{duerinckx2019scaling} (any dimension).

\medskip

\step1 Malliavin derivatives of $F_1$ and $F_2$.
\\
We claim that 
\begin{eqnarray}
D F_1& =& \int_{\mathbb{R}^d} \nabla u \cdot Da (\nabla \phi+e),\label{e.rep-fo-DF1}
\\
DF_2&=& \int_{\mathbb{R}^d} (\partial_i v e_j-\partial_j v e_i+\nabla w) \cdot Da (\nabla \phi+e),\label{e.rep-fo-DF2}
\end{eqnarray}
where $u,v,w$ solve
\begin{eqnarray}
-\nabla \cdot a \nabla u &=& \nabla \cdot g\label{e.pr-function1},
\\
\label{e.pr-function2}
-\triangle  v &=& \nabla \cdot g,\\
\label{e.pr-function3}
-\nabla \cdot a \nabla w& =& \nabla \cdot a (\partial_i v e_j-\partial_j v e_i).
\end{eqnarray}
(Notice that  these equations are all well-posed in $\dot H^1(\R^d)$ since $a$ is uniformly elliptic and bounded by assumption.)
Indeed $D F_1 = \int_{\R^d} \nabla D \phi \cdot g$ and, using the corrector equation~\eqref{CorectorEquation}, $D \phi$ solves
\[
-\nabla \cdot a \nabla D \phi \,=\, \nabla \cdot D a (\nabla \phi+e).
\]
Using \eqref{e.pr-function1}, we then obtain the representation formula \eqref{e.rep-fo-DF1}.
We now turn to $F_2$.
As for $F_1$ we have  
$D F_2 = \int_{\R^d} \nabla D \sigma_{ij} \cdot g$ and, using~\eqref{GaugeEquation}, $D \sigma_{ij}$ solves
\[
-\triangle D \sigma_{ij} \,=\, \partial_i D(a(\nabla \phi+e)\cdot e_j)-\partial_j D(a(\nabla \phi+e)\cdot e_i).
\]
Using first \eqref{e.pr-function2}, we have
\begin{eqnarray*}
D F_2 &=& \int (\partial_i v e_j-\partial_j v e_i) \cdot D(a(\nabla \phi+e)) 
\\
&=&\int (\partial_i v e_j-\partial_j v e_i) \cdot Da(\nabla \phi+e)+\int (\partial_i v e_j-\partial_j v e_i) \cdot a\nabla D\phi ,
\end{eqnarray*}
and we reformulate the last term using \eqref{e.pr-function3} to obtain \eqref{e.rep-fo-DF2}.

\medskip

Notice that our choice $a(G(x))=\exp(G(x))$ yields $D_za = a \delta(\cdot-z)$ for all $z\in \R^d$,
which we shall use in the estimates (again, the specific form is convenient but not essential).

\medskip

\step2 Application of the logarithmic-Sobolev inequality.
\\
By LSI, Step~1, and the identity $D_za = a \delta(\cdot-z)$, we have for all $q\ge 1$,
\[
\expec{|F|^{2q}}^\frac1q \,\lesssim\, q \expec{ \Big(\int_{\R^d} \Big(\int_{B(x)} (|\nabla u|+|\nabla v|+|\nabla w|) a (|\nabla \phi|+1)\Big)^2 dx \Big)^q}^\frac1q.
\]
We only focus on the term involving $\nabla w$, which is more involved since $w$ is obtained by solving two equations in a row.
By Cauchy-Schwarz',
\[
\Big(\int_{B(x)} |\nabla w| a (|\nabla \phi|+1)\Big)^2
\,\le\, \int_{B(x)} a|\nabla w|^2 \int_{B(x)} a(|\nabla \phi|^2+1).
\]
By duality in probability, this yields
\begin{multline*}
\expec{ \Big(\int_{\R^d} \Big(\int_{B(x)} |\nabla w| a (|\nabla \phi|+1)\Big)^2 dx \Big)^q}^\frac1q
\\
\le \, \sup_{X} \expec{\int_{\R^d} \Big(\int_{B(x)} a(|\nabla \phi|^2+1)\Big)\Big( \int_{B(x)} a|\nabla X w|^2\Big) dx },
\end{multline*}
where the supremum runs over random variables $X$ (which are thus independent of the space variable) such that $\expec{|X|^{2q'}}=1$.
We then set $\eta_\circ:=\frac{\theta}{(1+\theta)(q-1)}$, to the effect that $q'>1+\eta_\circ$ and 
$\frac{q'}{q'-(1+\eta_\circ)}=q(1+\theta)$, and use H\"older's inequality with exponents
$(\frac{q'}{q'-(1+\eta_\circ)},\frac{q'}{1+\eta_\circ})$, so that the above turns into
\begin{multline*}
\expec{ \Big(\int_{\R^d} \Big(\int_{B(x)} |\nabla w| a (|\nabla \phi|+1)\Big)^2 dx \Big)^q}^\frac1q
\\
\le \, \expec{\Big(\int_{B(0)} a(|\nabla \phi|^2+1)\Big)^{q(1+\theta)}}^\frac1{q(1+\theta)}
\sup_X\int_{\R^d}  \expec{\Big( \int_{B(x)} a|\nabla X w|^2\Big)^\frac{q'}{1+\eta_0} }^\frac{1+\eta_0}{q'}dx,
\end{multline*}
where we used the stationarity of $x\mapsto \int_{B(x)} a(|\nabla \phi|^2+1)$.
For convenience, we rewrite $1+\eta_\circ$ as $(1+\eta)^2$,
and apply Theorem~\ref{th:annealedmeyers} to \eqref{e.pr-function3}, which yields provided $
2q' \le 2+ \kappa$, 
\begin{equation*}
\int_{\R^d}  \expec{\Big( \int_{B(x)} a|\nabla X w|^2\Big)^\frac{q'}{(1+\eta)^2} }^\frac{(1+\eta)^2}{q'}dx
\,
\lesssim \, \zeta(\eta_0)\int_{\R^d}   \expec{\Big( \int_{B(x)} a|\nabla X v|^2\Big)^\frac{q'}{1+\eta} }^\frac{1+\eta}{q'}dx,
\end{equation*}
where $\zeta:t \mapsto t^{-\frac14}|\log t|^\frac12$ (since for $0<\eta_\circ<\frac12$, $\zeta(\eta)=\zeta(\sqrt{1+\eta_\circ}-1) \lesssim \zeta(\eta_\circ)$).
By H\"older's inequality with exponents $(\frac{1+\eta}{\eta},1+\eta)$, followed by the version of Theorem~\ref{th:annealedmeyers} for the Laplacian (see for instance \cite[Theorem~1.14]{{MR4619005}}) applied to \eqref{e.pr-function2} (with exponent $q' \lesssim 1$)
we further have  
\begin{eqnarray*}
\int_{\R^d}   \expec{\Big( \int_{B(x)} a|\nabla X v|^2\Big)^\frac{q'}{1+\eta} }^\frac{1+\eta}{q'}dx
&\le &\expec{a^{\frac {q'}\eta}}^\frac{\eta}{q'} \int_{\R^d}   \expec{\Big( \int_{B(x)} |\nabla X v|^2\Big)^{q'} }^\frac{1}{q'}dx
\\
&\lesssim & \expec{a^{\frac {q'}\eta}}^\frac{\eta}{q'} \expec{|X|^{2q'}}^\frac1{q'}  \int_{\R^d}  |g|^2
\\
&=&\expec{a^{\frac {q'}\eta}}^\frac{\eta}{q'}   \int_{\R^d}  |g|^2,
\end{eqnarray*}
where we used that $g$ is deterministic and $\expec{|X|^{2q'}}=1$.
The desired stochastic integrability comes from a direct calculation. On the one hand, if $q$ is large enough, $2\eta\sim\eta_0
\sim \frac{\theta}{1+\theta}q^{-1}$. On the other hand, using the Taylor expansion of the exponential and Gaussianity, we have
\begin{equation}\label{MomentA}
\expec{a^r}^\frac{1}{r}=\exp(\tfrac{\mathcal{C}(0)}{2}r)\quad\text{for any $r\geq 1$.}
\end{equation}
Finally, noticing in addition that $q'<2$, there is a $c>0$ depending on $\theta$ and $\mathcal{C}(0)$ such that
\begin{eqnarray*}
\zeta(\eta_0)\expec{a^{\frac {q'}\eta}}^\frac{\eta}{q'} =\eta_0^{-\frac 14}|\log \eta_0|^{\frac12} e^{\frac {q'}{2
\eta}}\, \le \, e^{c q}.
\end{eqnarray*}
This concludes the proof of the proposition.

\subsection{Proof of Proposition~\ref{prop:moment-en-dens}}
Wlog we may fix a direction $e$. We split the proof into four steps. For the proof of \eqref{MomentBoundGradientCor}, the core of the argument is to control the moments of $\rc$ in form of 
\begin{equation}\label{ControlMomentRC}
\expec{\rc^q}\le  e^{c q^2},
\end{equation}
for some $c>0$. We then deduce \eqref{MomentBoundGradientCor} using that the definition of $\rc$ in form of
\[\int_{B(0)} a(\vert\nabla \phi\vert^2+1)\lesssim \rc^d \fint_{B_{\rc}(0)} a\lesssim \rc^d,\]
where the latter is a consequence of $\rc\geq \ra$. We show \eqref{ControlMomentRC} in the first three  steps. In the last step, we prove that \eqref{Moment BoundFluxCor} is a consequence of \eqref{MomentBoundGradientCor}.

Since $\rc = \ra \mathds{1}_{\rc = \ra}+ \rc \mathds{1}_{\rc>\ra}$ and $\ra$ is well-controlled by Proposition~\ref{prop:mombd-ra}, it is enough to focus on $\rc \mathds{1}_{\rc>\ra}$. In what follows we use $\rc$ as a short-hand notation for $\rc \mathds{1}_{\rc>\ra}$, and recall that we have set $\rs(R):= \Big(R^{-\frac \e2} \sup_{B_R} (a+a^{-1})\Big)^2$. 

\medskip

\step1 Control of the level-sets of $\rc$.
\\
We claim that for all $0<\mu<1$ there exists $c>0$ such that for all dyadic $R$ and all $q\ge 1$, we have
\begin{multline}\label{e.Markov}
\expec{\mathds{1}(\rc=R)} \, \le\, c^q  R^{-q(d-\beta+2(1-\mu)-\e)}\expec{\rs(R)^q\rc^{q(d-\beta)}}
\\
+c^q R^{q\e} \expec{\rs(R)^q \Big(\fint_{B_R} \Big|\fint_{B_{R^\mu}(x)} \nabla \phi\Big|^2dx\Big)^q}.
\end{multline}
Assume that $\rc=R>\ra$. Then, by definition, we both have
\begin{eqnarray*}
\fint_{B_{R}} a|\nabla \phi|^2 &\le & C \fint_{B_{2R}} a,\\
\fint_{B_{\frac R2}} a|\nabla \phi|^2 &\ge & C \fint_{B_{R}} a.
\end{eqnarray*}
We now appeal to the Caccioppoli inequality~\eqref{e.pr-mey-1} applied to the corrector equation $-\nabla\cdot a \nabla \phi = \nabla \cdot a e$, which we rewrite in the form,
\[
\fint_{B_{\frac R2}} a |\nabla \phi|^2 \le C' \Big(\fint_{B_{R}} a+\inf_{c} \frac1{R^2} \fint_{B_R} a|\phi-c|^2 \Big),
\]
for some universal constant $C'$.
Provided we choose $C = 2C'$ (which completes Definition~\ref{def:rc}), this entails
\[
\inf_{c} \frac1{R^2} \fint_{B_R} a|\phi-c|^2 \gtrsim \fint_{B_R} a \gtrsim 1,
\]
using that $R \ge \ra$.

\medskip

Let $0<\mu <1$, and set $c_R := \fint_{B_R} \fint_{B_{R^\mu(x)}} \phi$.
By definition of $\rs$, we have
\[
1 \lesssim   \frac1{R^2} \fint_{B_R} a|\phi-c_R|^2  \le \sqrt{R^\e  \rs(R)} \frac1{R^2} \fint_{B_R} |\phi-c_R|^2.
\]
By the triangle inequality,
\begin{eqnarray*}
\frac1{R^2} \fint_{B_R} |\phi-c_R|^2&\lesssim& \frac1{R^2} \fint_{B_R} \Big| \phi-\fint_{B_{R^\mu}(x)} \phi\Big|^2dx+\frac1{R^2} \fint_{B_R} \Big|\fint_{B_{R^\mu}(x)} \phi-c_R\Big|^2.
\end{eqnarray*}
Using Poincar\'e's inequality (on domains of size $R^\mu$ for the first term, and on $B_R$ for the second term), this yields
\begin{eqnarray*}
\frac1{R^2} \fint_{B_R} |\phi-c_R|^2&\lesssim& R^{2(\mu-1)} \fint_{B_R} |\nabla \phi|^2+ \fint_{B_R} \Big|\fint_{B_{R^\mu}(x)} \nabla \phi\Big|^2.
\end{eqnarray*}
Whereas the second right-hand side term already has the correct form, we need to put back the weight on the first right-hand side term. For that we pay again $\sqrt{\rs(R) R^\e}$.
We have thus proved
\begin{eqnarray*}
1 &\lesssim &\rs(R) R^{\e}\Big( R^{2(\mu-1)}  + \fint_{B_R} \Big|\fint_{B_{R^\mu}(x)} \nabla \phi\Big|^2\Big)
\\
&\stackrel{\rc=R}=&\rs(R) R^{\e}\Big( R^{2(\mu-1)-d+\beta}\rc^{d-\beta}  + \fint_{B_R} \Big|\fint_{B_{R^\mu}(x)} \nabla \phi\Big|^2\Big).
\end{eqnarray*}
By Markov' inequality, this implies the claim.

\medskip

\step2 Hole-filling and control of the right-hand side of \eqref{e.Markov}.
\\
We claim that \eqref{e.Markov} entails for all $0<\theta<1$, all dyadic $R$ and all $q\ge 1$,
\begin{equation}\label{e.Markov2}
\expec{\mathds{1}(\rc=R)} \, \leq\,   e^{c_\theta q^2} \Big(R^{-q(d-\beta+2(1-\mu)-\e)} +
 R^{-q(d\mu-\e)}\Big) \expec{\rc^{q(d-\beta)(1+\theta)^3}}^\frac{1}{(1+\theta)^3},
\end{equation}
for some $c_{\theta}$ depending  on $\theta$. The first right-hand side term of \eqref{e.Markov2} comes from the first right-hand side term of \eqref{e.Markov} using H\"older's inequality and   the moment bound, for some $c>0$ and any $q\geq 1$,
\begin{equation}\label{AlgebraicMomentEllipticity}
\expec{\rs(R)^{q}}\leq  e^{c q^2},
\end{equation}
that we deduce from \eqref{MomentBoundEllipticityRadius}. The second right-hand side term is more subtle. By H\"older's inequality with exponents $(\frac{1+\theta}{\theta},1+\theta)$, Jensen's inequality, and stationarity of $x \mapsto \fint_{B_{R^\mu}(x)} \nabla \phi$ we have
\begin{eqnarray*}
\expec{ \rs(R)^q\Big(\fint_{B_R} \Big|\fint_{B_{R^\mu}(x)} \nabla \phi\Big|^2dx\Big)^q}
&\le& \expec{ \rs(R)^{q\frac{1+\theta}\theta}}^\frac{\theta}{1+\theta}
\expec{\Big|\fint_{B_{R^\mu}} \nabla \phi\Big|^{2q(1+\theta)}}^\frac1{1+\theta}\\
&\, \stackrel{\eqref{AlgebraicMomentEllipticity}}{\le} \,
& e^{c_{\theta,1} q^2
} \expec{\Big|\fint_{B_{R^\mu}} \nabla \phi\Big|^{2q(1+\theta)}}^\frac1{1+\theta}
\end{eqnarray*}
for a constant $c_{\theta,1}$. Then, by Proposition~\ref{prop:averages} applied to $g=|B_{R^\mu}|^{-1}\mathds 1_{B_{R^\mu}}$, we get
\begin{eqnarray*}
\expec{\Big|\fint_{B_{R^\mu}} \nabla \phi\Big|^{2q(1+\theta)}}^\frac1{q(1+\theta)}
&\le&  e^{c'_{\theta,2}q}\,\expec{\Big(\int_B a(|\nabla \phi|^2+1)\Big)^{q(1+\theta)^2}}^\frac1{q(1+\theta)^2} \int_{\mathbb{R}^d} |g|^2
\\
&=& e^{c'_{\theta,2}q}\,\expec{\Big(\int_B a(|\nabla \phi|^2+1)\Big)^{q(1+\theta)^2}}^\frac1{q(1+\theta)^2} R^{-d\mu},
\end{eqnarray*}
for some constant $c'_{\theta,2}$. It remains to control $\int_B a(|\nabla \phi|^2+1)$ by $\rc$, and we claim that $\int_B a(|\nabla \phi|^2+1)\lesssim \ra^d+\rc^{d-\beta}\ra^\beta$.
Indeed, if $\rc<\ra$, then 
\[
\int_B a(|\nabla \phi|^2+1)\le   \int_{B_\spadesuit} a(|\nabla \phi|^2+1) \lesssim 
 \int_{B_\spadesuit} a \lesssim \ra^d .
\]
Otherwise, $\rc \ge \ra$, and by the hole-filling estimate~\eqref{Lholefillingesti} in Corollary~\ref{Lholefilling},
\[
\int_B a(|\nabla \phi|^2+1)\lesssim \ra^d   \fint_{\Ba} a(|\nabla \phi|^2+1) \lesssim 
 \ra^d \big(\frac \rc \ra\big)^{d-\beta} \fint_{B_\spadesuit}  a(|\nabla \phi|^2+1)  \lesssim \rc^{d-\beta} \ra^\beta .
\]
All in all, this entails by H\"older's inequality with exponents $(\frac{1+\theta}\theta,1+\theta)$ and the moment bound \eqref{MomentNonSupEllRatio} on $\ra$,
\begin{equation*}
\expec{\Big|\fint_{B_{R^\mu}} \nabla \phi\Big|^{2q(1+\theta)}}^\frac1{q(1+\theta)}
\,\le \, e^{c_{\theta,2}q
} R^{-d\mu} \expec{\rc^{q(d-\beta)(1+\theta)^3}}^\frac1{q(1+\theta)^3}
\end{equation*}
for a constant $c_{\theta,2}$ and \eqref{e.Markov2} follows in combination with \eqref{e.Markov}.

\medskip

\step3 Buckling argument.
\\
By expressing the moments of $\rc$ using its level-sets we have for all $q\ge 1$,
and all $K\ge 1$,
\begin{eqnarray*}
\expec{\rc^{q(d-\frac\beta K)}}& \le& 1+\sum_{n=1}^\infty 2^{nq(d-\frac \beta K)} \expec{\mathds1(\rc=2^n)}
\\
&\stackrel{\eqref{e.Markov2}}\le & 1+   e^{c_{\theta}q^2} \expec{\rc^{q(d-\beta)(1+\theta)^3}}^\frac{1}{(1+\theta)^3}  
\\
&&\times \sum_{n=1}^\infty  (2^{-nq(\frac \beta K+2(1-\mu)-\beta-\e)}+ 2^{-nq(d(1-\mu)+\frac \beta K-\e)}) .
\end{eqnarray*}
We now choose our exponents.
We first fix $0<\mu<1$ so that $d(1-\mu)=\frac \beta2$, then set $\e:=\frac{\beta}{5d}$
and $\frac1K:=1-\frac1{5d}$ to the effect that
\[
 \frac12(2^{-nq(\frac \beta K+2(1-\mu)-\beta-\e)}+ 2^{-nq(d(1-\mu)+\frac \beta K-\e)}) \,\le \,2^{-nq \frac \beta{5d}}.
\]
With this choice, the series is summable and the above turns into 
\[
\expec{\rc^{q(d-\frac\beta K)}}\,  \le \,  1+  e^{cq^2} \expec{\rc^{(d-\beta)(1+\theta)^3q}}^\frac{1}{(1+\theta)^3}  .
\]
We may then absorb part of the right-hand side into the left-hand side by Young's inequality upon choosing $0<\theta <1$  so small but independent of $q$ that $(d-\beta)(1+\theta)^3 < d-\frac{\beta}{K}$ (which is possible since $K>1$), and the claimed moment bound {\color{red}\eqref{MomentBoundGradientCor}} follows.

\medskip

\step4 Control of $\nabla\sigma$. 
\\
We could control moment bounds of $\nabla\sigma$ by using Malliavin calculus and 
moment bounds on $\nabla \phi$. Here we directly apply standard Calder\'on-Zygmund estimates for the Laplacian (the scaling of the multiplicative constants follow from the Marcinkiewicz interpolation theorem \cite[Theorem $1.3.2$]{grafakos2008classical}, \cite{Marcinkiewicz}) to \eqref{GaugeEquation}: 
For all $R\geq 1$ and $q\geq 1$
\[
\Big(\fint_{B_R} \Big(\fint_{B(x)}\vert \nabla \sigma\vert^2\Big)^{q} \Big)^\frac1q \lesssim (q+q') \Big(\Big(\fint_{B_{2R}}\fint_{B(x)}\vert\nabla \sigma\vert^2\Big)^q+\fint_{B_{2R}}\Big(\fint_{B(x)}\vert a(\nabla\phi+e)\vert^2\Big)^q\Big)^\frac1q,
\]
By ergodicity and stationarity, as $R\uparrow +\infty$, each spatial average converges almost surely to the associated expectation, so that we derive 
\[
\expec{\Big(\fint_{B(0)}\vert \nabla \sigma\vert^2\Big)^{q}}^\frac1q \lesssim (q+q') \Big(\expec{\vert\nabla\sigma\vert^2}^q+\expec{\Big(\int_{B(0)}a^2(\vert \nabla\phi\vert^2+1)\Big)^q}\Big)^\frac1q.
\]
Finally, using 
\[
\int_{B(0)}a^2(\vert \nabla\phi\vert^2+1)\leq \Big(\sup_{B(0)} a\Big)\int_{B(0)} a(\vert\nabla\phi\vert^2+1),
\]
we obtain \eqref{Moment BoundFluxCor} from \eqref{MomentBoundGradientCor} and the local boundedness  of $a$ (see \eqref{e.a-bdd}).

\subsection{Arguments for Theorems~\ref{th:bd-corr} and~\ref{th:min-rad}}
We split the proof into two steps and we show separately Theorem \ref{th:bd-corr} and Theorem \ref{th:min-rad} using the standard strategy from \cite{GNO-reg,josien2020annealed}.

\medskip

\step1 Proof of Theorem \ref{th:bd-corr}. 
\\
Since $p_\diamond=d+1$, we have $\frac{p_\diamond-1}{2p_\diamond}>\frac{p_\diamond+1}{2p_\diamond}-\frac1d$, so that by Poincar\'e-Sobolev' inequality, 
\begin{eqnarray*}
\lefteqn{\Big(\fint_{B(x)} \Big|(\phi,\sigma)-\fint_{B(0)}(\phi,\sigma)\Big|^{\frac{2p_\diamond}{p_\diamond-1}}\Big)^\frac{p_\diamond-1}{2p_\diamond}}
\\
&\le&\Big|\fint_{B(x)}(\phi,\sigma)-\fint_{B(0)}(\phi,\sigma)\Big|+\Big(\fint_{B(x)} |(\phi,\sigma)-\fint_{B(x)}(\phi,\sigma)|^{\frac{2p_\diamond}{p_\diamond-1}}\Big)^\frac{p_\diamond-1}{2p_\diamond}
\\
&\lesssim&\Big|\fint_{B(x)}(\phi,\sigma)-\fint_{B(0)}(\phi,\sigma)\Big|+\Big(\fint_{B(x)} |\nabla (\phi,\sigma)|^{\frac{2p_\diamond}{p_\diamond+1}}\Big)^\frac{p_\diamond+1}{2p_\diamond}.
\end{eqnarray*}
Using H\"older's inequality, we have 
$$\Big(\fint_{B(x)} |\nabla (\phi,\sigma)|^{\frac{2p_\diamond}{p_\diamond+1}}\Big)^\frac{p_\diamond+1}{2p_\diamond}\leq \Big(\fint_{B(x)} a^{-p_\diamond}\Big)^{\frac{1}{2 p_{\diamond}}}\Big(\fint_{B(x)}a(\vert\nabla\phi\vert^2+1)\Big)^\frac{1}{2}+\Big(\fint_{B(x)}\vert \nabla\sigma\vert^2\Big)^{\frac{1}{2}},$$
that has the stochastic integrability \eqref{StoInteConstant} from Proposition \ref{prop:moment-en-dens} and \eqref{MomentA}. We then follow the standard proof to control $\Big|\fint_{B(x)}(\phi,\sigma)-\fint_{B(0)}(\phi,\sigma)\Big|$ in \cite{josien2020annealed}, and we show the argument for $\phi$ only (the proof for $\sigma$ follows the same way). We have the representation formula
\begin{equation}\label{RepresentationFormulaDiff}
\fint_{B(x)}\phi-\fint_{B(0)}\phi=\int_{\mathbb{R}^d}\nabla w\cdot \nabla \phi,
\end{equation}
where $w$ denotes the decaying solution of
\[-\triangle w=\frac{1}{\vert B(0)\vert}(\mathds{1}_{B(x)}-\mathds{1}_{B(0)}).\]
By classical potential theory, it holds
\[\Big(\int_{\mathbb{R}^d}\vert\nabla w\vert^2\Big)^{\frac{1}{2}}\lesssim \mu_d(|x|)=\left\{
\begin{array}{rcl}
\sqrt{|x|+1} &:&d=1,
\\
\log(|x|+2)^\frac12&:&d=2,
\\
1 &:&d>2.
\end{array}
\right.
\]
Since $w$ is deterministic, as a consequence of \eqref{RepresentationFormulaDiff} and \eqref{th:aver-corr-grad}, we finally obtain \eqref{MomentPhiItself}.

\medskip

\step2 Proof of Theorem \ref{th:min-rad}.
\\
By the layer-cake formula and since $r_\diamond$ is controlled in \eqref{MomentNonSupEllRatio}, it is sufficient to show that for any $\lambda>0$
\begin{equation}\label{TailProbaRstar}
\mathbb{P}\big(\{r_\star\geq \lambda\}\cap \{\lambda\geq r_\diamond\}\big)\le  e^{c q^2}\frac{\big(\mu_d(\lambda)\big)^q}{\lambda^q},
\end{equation}
for some $c>0$.
The estimate \eqref{TailProbaRstar} is simply obtained using the Definition \ref{DefRstarStatement} of $r_\star$ in form of 
$$\mathbb{P}\big(\{r_\star\geq \lambda\}\cap \{\lambda\geq r_\diamond\}\big)\leq \mathbb{P}\bigg(\Big(\fint_{B_{\frac{\lambda}{2}}}\Big\vert (\phi,\sigma)-\fint_{B_\frac{\lambda}{2}}(\phi,\sigma)\Big\vert^{\frac{2 p_\diamond}{p_\diamond-1}}\Big)^{\frac{p_\diamond-1}{2p_\diamond}}>\frac{\lambda}{2C}\bigg),$$
which, together with Markov's inequality and Theorem \ref{th:bd-corr}, gives \eqref{TailProbaRstar}.

\section{Large-scale regularity theory}\label{sec:LS}

\subsection{Large-scale Lipschitz estimates}
The proof is based on a post-processing of \cite[Theorem $2$]{bella2018liouville}. Since the proof is very similar to that of \cite[Corollary $3$]{GNO-reg}, we only highlight the main steps of the argument. The starting point is
\begin{equation}\label{ExcessDecayBFO}
\begin{aligned}
\sup_{r\in [\ru,R]} &\tfrac{1}{r^{2\alpha}}\mathrm{Exc}(\nabla u+g,r)\lesssim\tfrac{1}{R^{2\alpha}}\mathrm{Exc}(\nabla u+g,R)\\
&+\sup_{r\in [\ru,R]}\tfrac{1}{r^{2\alpha}}\fint_{B_r}\Big(\big(g-\fint_{B_r} g\big)\cdot a\big(g-\fint_{B_r} g\big)+\big(h-\fint_{B_r} h\big)\cdot a^{-1}\big(h-\fint_{B_r} h\big)\Big),
\end{aligned}
\end{equation}
where the excess $\mathrm{Exc}$ is defined for any $\rho>0$ by
$$\mathrm{Exc}(\nabla u+g,\rho):=\inf_{\xi\in\mathbb{R}^d}\fint_{B_\rho}\big(\nabla u-(\xi+\nabla\phi_\xi)\big)\cdot a\big(\nabla u-(\xi+\nabla\phi_\xi)\big).$$
The estimate \eqref{ExcessDecayBFO} can be obtained following the lines of \cite[Step $1$ p. 135]{GNO-reg} that is based on energy estimates and the excess decay \cite[Theorem $2$]{bella2018liouville}. We can then post-process \eqref{ExcessDecayBFO} following the lines of \cite[Step $2$ p. 136]{GNO-reg} where the additional main ingredient is the non-degeneracy of the correctors which reads : for any $\rho\geq \ru$
\begin{equation}\label{NonDegeneracyCorrector}
\fint_{B_\rho}\big(\nabla\phi_\xi+\xi\big)\cdot a\big(\nabla\phi_\xi+\xi\big)\gtrsim \vert \xi\vert^2.
\end{equation}
To see \eqref{NonDegeneracyCorrector}, first note that from the Definition \ref{prop:mombd-ra} of $r_\diamond$ we have by H\"older and Poincar\'e's inequalities
$$\Big(\fint_{B_\rho}\big(\nabla\phi_\xi+\xi\big)\cdot a\big(\nabla\phi_\xi+\xi\big)\Big)^{\frac{1}{2}}\gtrsim \frac{1}{\rho}\Big(\fint_{B_\rho}\Big\vert\phi_\xi+\xi\cdot x-\fint_{B_\rho}\phi_\xi\Big\vert^{\frac{2p_\diamond}{p_\diamond+1}}\Big)^{\frac{p_\diamond+1}{2 p_\diamond}}.$$
Then, using the triangle inequality in form of 
$$\frac{1}{\rho}\Big(\fint_{B_\rho}\Big\vert\phi_\xi+\xi\cdot x-\fint_{B_\rho}\phi_\xi\Big\vert^{\frac{2p_\diamond}{p_\diamond+1}}\Big)^{\frac{p_\diamond+1}{2 p_\diamond}}\geq \Big\vert  C_d\vert \xi\vert- \frac{1}{\rho}\Big(\fint_{B_\rho}\Big\vert\phi_\xi-\fint_{B_\rho}\phi_\xi\Big\vert^{\frac{2p_\diamond}{p_\diamond+1}}\Big)^{\frac{p_\diamond+1}{2 p_\diamond}}\Big\vert,$$
for some $C_d>0$, together with the definition of $\ru$ in form of 
$$\frac{1}{\rho}\Big(\fint_{B_\rho}\Big\vert\phi_\xi-\fint_{B_\rho}\phi_\xi\Big\vert^{\frac{2p_\diamond}{p_\diamond+1}}\Big)^{\frac{p_\diamond+1}{2 p_\diamond}}\leq\frac{1}{\rho}\Big(\fint_{B_\rho}\Big\vert\phi_\xi-\fint_{B_\rho}\phi_\xi\Big\vert^{\frac{2p_\diamond}{p_\diamond-1}}\Big)^{\frac{p_\diamond-1}{2 p_\diamond}} \leq \frac{\vert \xi\vert}{C},$$
we obtain \eqref{NonDegeneracyCorrector} (up to increasing the value of $C$).

\subsection{Quenched and annealed Calder\'on-Zygmund estimates}
The proof of Theorem \ref{th:quenchedCZ} is based on the Schauder regularity theory in Proposition \ref{SchauderTheoryLarge} and a refined dual version of the Calder\'on-Zygmund lemma due to Shen \cite[Theorem $3.2$]{Shen-07}. We then obtain Theorem \ref{th:annealedCZ} as a post-processing of Theorem \ref{th:quenchedCZ} using the moment bound on $\ru$ in Theorem \ref{th:min-rad}. Since the proofs of Theorem \ref{th:quenchedCZ} and Theorem \ref{th:annealedCZ} follow almost line by line the proofs of \cite[Proposition $6.4$]{DO-20} and \cite[Theorem $6.1$]{DO-20} respectively, we leave the details to the reader.

\appendix

\section{Proof of Proposition~\ref{prop:mombd-ra}}

\subsection{A result on stochastic integrability}

We start by showing the following lemma on the characterization of super-algebraic stochastic integrability appearing in Proposition~\ref{prop:mombd-ra}:
\begin{lem0}\label{ap: sto_int}
Let $\alpha>1$. Given a positive random variable $X$, the following statements are equivalent:
\begin{enumerate}
\item\label{ap: sto_int-exp} 
There exists a constant $C>0$ such that
\begin{equation*}
\expec{\exp(\frac 1C \log^\alpha(1+X))} < +\infty.
\end{equation*}
\item\label{ap: sto_int-tail}
 There exists a constant $C>0$ such that for $x$ large enough,
\begin{equation*}
\Pm(X\geq x) \,\le\, \exp(-\frac 1C \log^\alpha(1+x)).
\end{equation*}
\item\label{ap: sto_int-mmt}
 There exists a constant $C>0$ such that for $p$ large enough,
\begin{equation*}
\expec{X^p}\,\le\, \exp(Cp^{\frac{\alpha}{\alpha-1}}).
\end{equation*}
\end{enumerate}
Here the constants $C$ might be different. In particular, if $\alpha=2$, all the exponents are $2$.
\end{lem0}
We prove \eqref{ap: sto_int-exp} $\Rightarrow$ \eqref{ap: sto_int-tail}$\Rightarrow$ \eqref{ap: sto_int-mmt}$\Rightarrow$ \eqref{ap: sto_int-exp}.
\\
\eqref{ap: sto_int-exp} $\Rightarrow$ \eqref{ap: sto_int-tail}: Direct application of Chebychev's inequality.
\\
\eqref{ap: sto_int-tail} $\Rightarrow$ \eqref{ap: sto_int-mmt}: For $A>0$ large enough to be chosen later, we have by a change of variable
\begin{eqnarray*}
\expec{X^p}  & \le &  A^p+p\int_A^{+\infty} \Pm(X>x) x^{p-1}dx\\
& \lesssim &  A^p + p\int_A^{+\infty} \exp\left(-\frac 1C \log^\alpha(1+x)\right) x^{p-1} dx\\
& \lesssim & A^p +p\int_{\log^\alpha(1+A)}^{+\infty}e^{-\frac 1C t} e^{p t^{\frac 1\alpha}}dt.
\end{eqnarray*}
We then choose $A$ so large that for any $t>A$ one has $-\frac1C t+pt^{\frac1\alpha}<-\frac 1{2C} t$. Since $\alpha>1$, we can take  $A=e^{(2Cp)^{\frac1{\alpha-1}}}$. Hence by changing the constant $C$,
\begin{eqnarray*}
\expec{X^p}  & \lesssim &  A^p+pC \,\le\,e^{C p^{\frac\alpha{\alpha-1}}}.
\end{eqnarray*}
\\
\eqref{ap: sto_int-mmt} $\Rightarrow$ \eqref{ap: sto_int-exp}: Denote by $C'$ the constant in \eqref{ap: sto_int-mmt}. By Chebychev's inequality with power $p_n$ yet to be chosen,
\begin{eqnarray*}
\expec{\exp(\frac 1C \log^\alpha(1+X))} & = & \int_0^{+\infty} \Pm(\exp(\frac 1C \log^\alpha(1+X))>t) dt\\
& = & \int_0^1 1 dt +\left(\sum_{n\in\N}\int_{e^n}^{e^{n+1}}\right)\Pm(X>\exp\left((C\log t)^{\frac 1\alpha}\right))dt\\
& \le& 1+ \sum_{n\in\N}\int_{e^n}^{e^{n+1}}\Pm(X>\exp\left((C\log t)^{\frac 1\alpha}\right))dt\\
& \le & 1+ \sum_{n\in\N}e^{n+1}\exp\left(-C^{\frac1\alpha} p_n n^{\frac 1\alpha}\right) \exp(C'p_n^{\frac \alpha{\alpha-1}}).
\end{eqnarray*}
With $p_n=n^{1-\frac1\alpha}$ and $C$ large enough, the sum converges, and the claim follows.

\subsection{Proof of Proposition~\ref{prop:mombd-ra}}

We split the proof into two steps. We first estimate $\ra$, and then turn to $\rs$.

\medskip

\step{1} Estimate of $\ra$.\\
We first introduce a stationary random radius $\tra$ defined by 
\begin{multline*}
\tra(x):= \inf_{r \ge 1 \text{ dyadic}} \Big\{\forall \rho \ge r \text{ dyadic}, \, 
\\
\tfrac1{C_d}\expec{a^{p_\diamond}+a^{-p_\diamond}} \le\fint_{B_r(x)} a^{p_\diamond}+a^{-p_\diamond} \le C_d \expec{a^{p_\diamond}+a^{-p_\diamond}}\Big\},
\end{multline*}
where $C_d$ will be chosen later.
For all $r>0$, we set 
\begin{equation*}
X_r\colon=\fint_{B_r} a^{\pa}+a^{-\pa}.
\end{equation*}
The Malliavin derivative of $X_r$ is $D X_r= \pa |B_r|^{-1}(a^{\pa}+a^{-\pa})\mathds{1}_{B_r}$. Hence the moment bound for all $q\ge 1$
\begin{equation}\label{ap:sto_int-hyp}
\expec{|X-\E X|^q}^{\frac 1q} \,\lesssim\, \sqrt{q}\expec{\|DX\|_{L^2(\R^d)}^{q}}^{\frac 1q}
\end{equation}
(which is a consequence of the logarithmic-Sobolev inequality) 
combined with Minkowski's inequality implies
\begin{equation*}
\expec{|X_r-\E X_r|^q}^{\frac 1q} \,\lesssim\, \sqrt{q}\expec{\|DX\|_{L^2(\R^d)}^{q}} ^{\frac 1q}\,\lesssim\, \sqrt{q} |B_r|^{-\frac 12} \expec{a^{\pa q}}^{\frac 1q}.
\end{equation*}
By definition of the random field $a$, $\expec{a^p}=\exp(\frac{\mathcal{C}(0)}{2}p^2)$,
so for suitable positive constants $c$ and $C$ (depending on $\pa$),
\begin{equation*}
\expec{|X_r-\E X_r|^q} \,\le\, C^q  q^q |B_r|^{-\frac q2} \exp(\frac12 \pa^2 q^2)\,\le\, r^{-\frac d2 q}\exp(cq^2).
\end{equation*}
Since the definition of $\tra$ only involves dyadic radii, we have
\begin{equation*}
\Pm(\tra>x)\le \sum_{r=2^n>x} \Pm\left(|X_r-\E X_r|>(1-C_d^{-1}) \E X_r\right).
\end{equation*}
By applying Chebychev's inequality with power $q_n=\frac{d \log2 }{4c}n$ to each term in the sum, we obtain
\begin{equation*}
\Pm(\tra>x)\,\lesssim\, \sum_{n>\log_2 x} 2^{-\frac d2 nq_n} \exp(cq_n^2)\,\lesssim\, \sum_{n>\log_2 x} \exp(-c n^2)\,\le\, e^{-c\log^2x},
\end{equation*}
here the constant $c$ varies but only depends on $d$ and $\pa$. This implies the claimed stochastic integrability of $\tra$ by Lemma~\ref{ap: sto_int}.

To pass from dyadic radii to general radii, it  suffices to multiply $\tra$ by $2$. In fact, for a radius $r\in [r_0,2r_0]$,
\begin{equation*}
 \fint_{B_{r}} \left|a^\pa + a^{-\pa}-\expec{ a^\pa + a^{-\pa}}\right|\le 2^d\fint_{B_{r_0}} \left|a^\pa + a^{-\pa}-\expec{ a^\pa + a^{-\pa}}\right|
\end{equation*}
As a result, for any $r>2\tra\ge 2^{\lceil \log_2 \tra\rceil}$, by comparing the integral with the average on the ball with the nearest dyadic radius (larger than $\ra$), 
\begin{equation}\label{ap: non_dyadic_radii}
\fint_{B_{r}} \left|a^\pa + a^{-\pa}-\expec{ a^\pa + a^{-\pa}}\right| \le (1-C_d^{-1})2^d  \fint_{B_{r}} \left|a^\pa + a^{-\pa}-\expec{ a^\pa + a^{-\pa}}\right| .
\end{equation}

We then pick 
\begin{equation*}
\underline{r}_{\diamond}(x,\epsilon)= \inf_{y} \tra(y,\epsilon)+\frac 18 |x-y|,
\end{equation*}
with
\begin{equation*}
\tilde r_{\diamond}(x,\epsilon):= \inf \left\{r\colon \forall \rho > r , \fint_{B_{\rho}(x)} \left| (a^\pa + a^{-\pa})-\expec{ a^\pa + a^{-\pa}}\right| \le \epsilon \expec{ a^\pa + a^{-\pa}} \right\}.
\end{equation*}
 By  construction, $\ura$ is the maximal $\frac 18$-Lipschitz field smaller than $\tra$, so $\ura(x,\epsilon)\le\tra(x,\epsilon)$. If $\ura(x,\epsilon)\le R$, by definition there is a $y\in \R^d$ such that $|y-x|\leq 8R$ and $\tra(y,\epsilon)\le R$. This implies that for any $\rho>R$, $B_{\rho}(x)\subset B_{9\rho}(y)$. Thus
\begin{eqnarray*}
\fint_{B_{\rho}(x)} \left| (a^\pa + a^{-\pa})-\expec{ a^\pa + a^{-\pa}}\right| & \le & 9^d  \fint_{B_{9\rho}(x)} \left| (a^\pa + a^{-\pa})-\expec{ a^\pa + a^{-\pa}}\right|\\
& \le & 9^d\epsilon \expec{ a^\pa + a^{-\pa}} .
\end{eqnarray*}
Again by definition, $\tra(x,9^{d}\epsilon)\leq R$, which yields $\tra\left(x,9^{d}\epsilon\right)\le \ura(x,\epsilon).$ So if we define $\ra$ to be the minimal $\frac 18$-Lipschitz random field with the desired large scale regularity property, since $\tra(x,\frac12) \le \ura(x,\frac12 9^{-d})$ and the latter one is $\frac 18$-Lipschitz, by minimality
\begin{equation*}
\ra(x) \le \ura(x, \frac12 9^{-d}) \le \tra(x,\frac12 9^{-d}).
\end{equation*}
As a result, in  view of \eqref{ap: non_dyadic_radii}, if we pick $C_d$ such that $(1-C_d^{-1}) 2^d=\frac12 9^{-d}$, $\ra$ will have the desired stochastic integrability.

\medskip

\step{2} Estimate of $\rs$.\\
Since the laws of $a$ and $a^{-1}$ are the same, it is enough to treat positive powers of $a$.
We first claim that it suffices to control $\rs(1)$. In fact, if 
\begin{equation*}
\expec{\exp\left(\frac 1{C} \log^2(1+\rs(1))\right)} < +\infty
\end{equation*}
holds, then since one can cover $B_R$ with $c_d R^d$ balls $B_i$ of radius $1$,  by stationarity of $a$
\begin{eqnarray*}
\Pm(\rs(R)>r) & \leq & \Pm\left(\exists 1 \le i \le c_d R^d \textup{ s.t. } \sup_{B_i} a >R^{\varepsilon}r\right)\\
& \le & \sum_{i=1}^{c_d R^d}\Pm\left( \sup_{B_i} a >R^{\varepsilon}r\right)\\
& \lesssim & c_d R^d  \exp(-\frac 1C \log^2(1+R^\varepsilon r))\\
& \lesssim &c_d R^d  \exp(-\frac 1C \log ^2(R^\varepsilon)) \exp(-\frac 1C \log^2(1+ r))
\\
&\lesssim & \exp(-\frac 1C \log^2(1+ r))
\end{eqnarray*}
since $R \mapsto R^d  \exp(-\frac 1C \log ^2(R^\varepsilon))$ is bounded on $\R_+$. This entails the claimed stochastic integrability of $\rs(R)$ by Lemma~\ref{ap: sto_int}.

\medskip

We now estimate $\rs(1)$. Under the hypothesis that the covariance function $\Cc$ is $2\gamma$-H\"older continuous at $0$, we have the following regularity estimate: for a constant $c$, and all $q\ge 1$,
\begin{equation}\label{HolderregularityCoef} 
\expec{|a(x)-a(y)|^q}\, \lesssim \, |x-y|^{\gamma q} e^{cq^2}.
\end{equation}
Indeed, by the triangle inequality and the growth of moments of Gaussian variables,
\begin{eqnarray*}
\expec{|a(x)-a(y)|^q}^{\frac 1q}& = & \expec{e^{qG(y)}\left|1-e^{G(x)-G(y)}\right|^q}^{\frac 1q}\\
& \le & \expec{e^{2qG(y)}}^{\frac 1{2q}}\left(\sum_{n=1}^{+\infty} \frac {1}{n!}\expec{|G(x)-G(y)|^{2nq}}^{\frac 1{2q}}\right)\\
& \, \lesssim \, & e^{q}\sum_{n=1}^{+\infty} \frac {1}{n!}(2nq)^{\frac n2}\expec{|G(x)-G(y)|^{2}}^{\frac n{2}}\\
& \le & e^q\sum_{n=1}^{+\infty} ([\frac n2]!)^{-1}(4q)^{\frac n2}(2-2\Cc(x-y))^{\frac n{2}}\\
& \, \lesssim \,  & e^q\sum_{n=1}^{+\infty} (n!)^{-1}(c q)^{n}|x-y|^{n\gamma}\leq e^{cq} |x-y|^{\gamma},
\end{eqnarray*}
where $c$ is a constant only depending on the H\"older continuity of $\Cc$. 
\\
It remains to control $\rs(1)$ by a Kolmogorov criterion-type argument.
We now work in the unit cube $Q_1$ for convenience. Denote $\left((Q^n_k)_{,1\le k \le 2^{nd}}\right)_{n\in\N}$ the family of coverings of $Q_1$ by disjoint dyadic cubes of side length $2^{-n}$ such that $Q_1=Q^0_1$. For $Q_{1},Q_{2}\in (Q_k^n)_{n\in\N,1\le k \le 2^n}$ with $Q_1\subset Q_2$, by \eqref{HolderregularityCoef} and the Minkowski inequality we have
\begin{equation*}
\expec{\left|\fint_{Q_1}a-\fint_{Q_2}a\right|^q}^{\frac 1q} \le \expec{\left|\fint_{Q_1}\fint_{Q_2}|a(x)-a(y)|dxdy\right|^q}^{\frac 1q}\, \lesssim \, e^{cq}(\diam Q_2)^\gamma,
\end{equation*}
which implies, by replacing the first $\sup$ by a sum,
\begin{equation}\label{e.app-Kolmo}
\expec{\sup_{Q^n_\cdot}\sup_{k\colon Q^{n+1}_k\subset Q^n_{\cdot}}\left|\fint_{Q^n_\cdot}a-\fint_{Q^{n+1}_k}a\right|^q}^{\frac 1q}  \, \lesssim \,   e^{cq}2^{-n(\gamma-\frac dq)}.
\end{equation}
For all $n$, approximate $a$ by its local averages $a_n$ at scale $2^{-n}$, that is,  
\begin{equation*}
a_n :=  \sum_{k=1}^{2^n} \mathds{1}_{Q^n_k} \fint_{Q^n_k}a .
\end{equation*}
For $q$ such that  $\gamma-\frac dq>0$, \eqref{e.app-Kolmo} implies that $a_n$ is a Cauchy sequence in $L^q(d\mathbb P,L^\infty(Q_1))$. Since $a_n$ converges to $a$ almost surely almost everywhere by the Lebesgue differentiation theorem, we obtain
\begin{eqnarray}
\expec{\left(\sup_{Q_1} a\right)^q}^{\frac 1q} & \le & \sum_{n=1}^{\infty}\expec{\sup_{Q_1}|a_n-a_{n+1}|^q}^{\frac 1q} + \expec{\sup_{Q_1}|a_0|^q}^{\frac 1q}\nonumber \\
&\,\lesssim \,& \sum_{n=0}^{+\infty} e^{cq}2^{-n(\gamma-\frac dq)}+\expec{\left|\fint_{Q_1}a\right|^q}^{\frac1q} \, \le \, e^{cq}.\label{e.a-bdd}
\end{eqnarray}
For all $q>\frac {2d}\gamma$, this yields a moment bound, which, according to Lemma~\ref{ap: sto_int}, entails the claimed stochastic integrability of $\rs(1)$. 

\section*{Acknowledgements}

NC has received funding from the European Research Council (ERC) under the European Union's Horizon 2020 research and innovation programme (Grant agreement No 948819). The two last authors acknowledge support from the European Research Council (ERC) under the European Union's Horizon 2020 research and innovation programme (Grant Agreement No 864066).


\def\cprime{$'$}

\end{document}